\documentclass[a4paper,oneside,11pt]{article}%
\usepackage{makeidx}
\usepackage[english]{babel}
\usepackage{amsmath}
\usepackage{amsfonts}
\usepackage{amssymb}
\usepackage{stmaryrd}
\usepackage{graphicx}
\usepackage{mathrsfs}
\usepackage[colorlinks,linkcolor=red,anchorcolor=blue,citecolor=blue,urlcolor=blue]{hyperref}
\usepackage[symbol*,ragged]{footmisc}

\allowdisplaybreaks[4]
\providecommand{\U}[1]{\protect\rule{.1in}{.1in}}
\providecommand{\U}[1]{\protect \rule{.1in}{.1in}}

\newtheorem{theorem}{Theorem}[section]

\newtheorem{lemma}[theorem]{Lemma}

\newenvironment{proof}[1][Proof]{\noindent \textbf{#1.} }{\  \rule{0.5em}{0.5em}}
\numberwithin{equation}{section}

\begin{document}

\title{A Diophantine inequality with four prime variables }

\author{Min Zhang\footnotemark[1]\,\,\,\, \, \& \,\,Jinjiang Li\footnotemark[2] \vspace*{-4mm} \\
$\textrm{\small Department of Mathematics, China University of Mining and Technology}^{*\,\dag}$
                    \vspace*{-4mm} \\
     \small  Beijing 100083, P. R. China  }

\footnotetext[2]{Corresponding author. \\
    \quad\,\, \textit{ E-mail addresses}:
     \href{mailto:min.zhang.math@gmail.com}{min.zhang.math@gmail.com} (M. Zhang),
     \href{mailto:jinjiang.li.math@gmail.com}{jinjiang.li.math@gmail.com} (J. Li).   }

\date{}
\maketitle

{\textbf{Abstract}}: Let $N$ be a sufficiently large real number. In this paper, it is proved that, for $1<c<\frac{1193}{889}$, the following
Diophantine inequality
\begin{equation*}
    \big|p_1^c+p_2^c+p_3^c+p_4^c-N\big|<\log^{-1}N
\end{equation*}
is solvable in prime variables $p_1,p_2,p_3,p_4$, which improves the result of Mu \cite{Mu-Quanwu-2015}.


{\textbf{MR(2010) Subject Classification}}: 11J25, 11P32, 11P55, 11L07, 11L20

\section{Introduction and main result}
Let $k\geqslant1$ be a fixed integer and $N$ a sufficiently large integer. The famous Waring--Goldbach problem is to study the solvability of the following Diophantine equality
\begin{equation}\label{Waring-Goldbach-general}
  N=p_1^k+p_2^k+\cdots+p_r^k
\end{equation}
in prime variables $p_1,p_2,\dots,p_k$. For linear case, in 1937, Vinogradov \cite{Vinogradov-1937} proved that every sufficiently large odd integer $N$
can be written as the sum of three primes. For $k=2$, in 1938, Hua \cite{Hua-1938} proved that the equation (\ref{Waring-Goldbach-general}) is solvable for $r=5$ and sufficiently large integer $N$ satisfying $N\equiv 5\pmod {24}$.

In 1952, Piatetski-Shapiro \cite{Piatetski-Shapiro-1952} studied the following analogue of the Waring--Goldbach problem:
Suppose that $c>1$ is not an integer, $\varepsilon$ is a small positive number, and $N$ is a sufficiently large real number. Denote
by $H(c)$ the smallest natural number $r$ such that the following Diophantine inequality
\begin{equation}\label{Diophantine-inequality-general}
  |p_1^c+p_2^c+\cdots+p_r^c-N|<\varepsilon
\end{equation}
is solvable in primes $p_1,p_2,\dots,p_r$, then it was proved in \cite{Piatetski-Shapiro-1952} that
\begin{equation*}
  \limsup_{c\to+\infty}\frac{H(c)}{c\log c}\leqslant4.
\end{equation*}
Also, in \cite{Piatetski-Shapiro-1952}, Piatetski--Shapiro considered the case $r=5$ in (\ref{Diophantine-inequality-general}) and proved that $H(c)\leqslant5$
for $1<c<3/2$. Later, the upper bound $3/2$ for $H(c)\leqslant5$ was improved successively to
\begin{equation*}
  \frac{14142}{8923},\quad \frac{1+\sqrt{5}}{2}, \quad\frac{81}{40}, \quad\frac{108}{53},\quad 2.041, \quad\frac{665576}{319965}
\end{equation*}
by Zhai and Cao \cite{Zhai-Cao-2003}, Garaev \cite{Garaev-2003}, Zhai and Cao \cite{Zhai-Cao-2007}, Shi and Liu \cite{Shi-Liu-2013},
Baker and Weingartner \cite{Baker-Weingartner-2013}, Zhang and Li \cite{Zhang-Li-2018}, respectively. 

From these results and Goldbach--Vinogradov theorem, it is reasonable to conjecture that if $c$ is
near to $1$, then the Diophantine inequality (\ref{Diophantine-inequality-general}) is solvable for $r=3$. This conjecture was first 
established by Tolev \cite{Tolev-PhD-thesis} for $1<c<\frac{27}{26}$. Since then, the range of $c$ was enlarged to
\begin{equation*}
  \frac{15}{14},\quad \frac{13}{12},\quad \frac{11}{10},\quad \frac{237}{214},\quad \frac{61}{55},\quad \frac{10}{9},\quad \frac{43}{36}
\end{equation*}
by Tolev \cite{Tolev-1992}, Cai \cite{Cai-1996}, Cai \cite{Cai-1999} and Kumchev and Nedeva \cite{Kumchev-Nedeva-1998} independently, Cao and Zhai \cite{Cao-Zhai-2002},
Kumchev \cite{Kumchev-1999}, Baker and Weingartner \cite{Baker-Weingartner-2014}, Cai \cite{Cai-2018}, successively and respectively. 

Combining Tolev's method and  the techniques of estimates on exponential sums of Fouvry and Iwaniec, in 2003, Zhai and Cao \cite{Zhai-Cao-2003} proved that
$H(c)\leqslant4$ for $1<c<\frac{81}{68}$. Later, the range of $c$ for $H(c)\leqslant4$ was enlarged to $1<c<\frac{97}{81}$ by
Mu \cite{Mu-Quanwu-2015}.

 In this paper, motivated by \cite{Cai-2018}, we shall continue to improve the result of Mu and establish the following theorem.

\begin{theorem}\label{Theorem-four-primes-variables}
   Suppose that $1<c<\frac{1193}{889}$, then for any sufficiently large real number $N$, the following Diophantine inequality
\begin{equation}\label{qu-5}
   \big|p_1^c+p_2^c+p_3^c+p_4^c-N\big|<\log^{-1}N
\end{equation}
is solvable in primes $p_1,p_2,p_3,p_4$.
\end{theorem}

\smallskip 
\textbf{Remark } In order to compare our result with the results of Mu \cite{Mu-Quanwu-2015} and
Zhai and Cao \cite{Zhai-Cao-2003}, we list the numerical result as follows
\begin{equation*}
   \frac{1193}{889}=1.341957255\cdots;\,\quad\frac{97}{81}=1.197530864\cdots;\,\,\quad  \frac{81}{68}=1.191176471\cdots.
\end{equation*}

\smallskip
\textbf{Notation.}
Throughout this paper, we suppose that $1<c<\frac{1193}{889}$. Let $p$, with or without subscripts, always denote a prime number. $\eta$
always denotes an arbitrary small positive constant, which may not be the same at different occurrences; $N$ always denotes a sufficiently large real number.
As usual, we use $\Lambda(n)$ to denote von Mangoldt's function; $e(x)=e^{2\pi i x}$; $f(x)\ll g(x)$ means that $f(x)=O(g(x))$; $f(x)\asymp g(x)$ means that $f(x)\ll g(x)\ll f(x)$.

We also define
\begin{align*}
 & \,\,  X=\frac{1}{2}\bigg(\frac{2N}{5}\bigg)^{\frac{1}{c}},\qquad \varepsilon=\log^{-2}X,\qquad K=\log^{10}X, \qquad \tau=X^{1-c-\eta},   \\
 & \,\,  S(x)=\sum_{X<p\leqslant2X}(\log p)e\big(p^{c}x\big), \quad  I(x)=\int_X^{2X}e(t^cx)\mathrm{d}t,\quad\mathcal{T}(x)=\sum_{X<n\leqslant2X}e(n^cx).
\end{align*}

\section{Preliminary Lemmas}
In this section, we shall give some preliminary lemmas, which are necessary in the proof of Theorem \ref{Theorem-four-primes-variables}.

\begin{lemma}\label{xiaobei-lemma}
   Let $a,b$ be real numbers, $0<b<a/4,$ and let $r$ be a positive integer. Then there exists a function $\phi(y)$ which is $r$ times
   continuously differentiable and such that
  \begin{equation*}
    \left\{
      \begin{array}{cll}
          \phi(y)=1,    & &  \textrm{if \quad} |y|\leqslant a-b, \\
          0<\phi(y)<1,  & &  \textrm{if \quad} a-b<|y|< a+b, \\
          \phi(y)=0,    & &  \textrm{if \quad} |y|\geqslant a+b,
      \end{array}
    \right.
  \end{equation*}
  and its Fourier transform
   \begin{equation*}
      \Phi(x)=\int_{-\infty}^{+\infty} e(-xy)\phi(y)\mathrm{d}y
   \end{equation*}
   satisfies the inequality
   \begin{equation*}
      \left|\Phi(x)\right|\leqslant\min\left(2a,\frac{1}{\pi|x|},\frac{1}{\pi|x|}\left(\frac{r}{2\pi|x|b}\right)^r\right).
   \end{equation*}
\end{lemma}
\begin{proof}
 See Piatetski--Shapiro~\cite{Piatetski-Shapiro-1952} or Segal~\cite{Segal-1933-1}.  $\hfill$
\end{proof}

\begin{lemma}\label{Fouvry-Iwaniec-chafen}
Let $\mathcal{L},\mathcal{Q}\geqslant1$ and $z_\ell$ be complex numbers. Then we have
\begin{equation*}
 \Bigg|\sum_{\mathcal{L}<\ell\leqslant2\mathcal{L}}z_\ell\Bigg|^2\leqslant\bigg(2+\frac{\mathcal{L}}{\mathcal{Q}}\bigg)
 \sum_{|q|<\mathcal{Q}}\bigg(1-\frac{|q|}{\mathcal{Q}}\bigg)
 \sum_{\mathcal{L}<\ell+q,\ell-q\leqslant2\mathcal{L}}z_{\ell+q}\overline{z_{\ell-q}}.
\end{equation*}
\end{lemma}
\begin{proof}
 See Lemma 2 of Fouvry and Iwaniec \cite{Fouvry-Iwaniec-1989}.  $\hfill$
\end{proof}

\begin{lemma}\label{Titchmarsh-42}
  Let $f(x)$ be a real differentiable function such that $f'(x)$ is monotonic, and $|f'(x)|\geqslant m>0$, throughout the interval $[a,b]$. Then we have
\begin{equation*}
   \bigg|\int_a^be^{if(x)}\mathrm{d}x\bigg|\ll \frac{4}{m}.
\end{equation*}
\end{lemma}
\begin{proof}
 See Lemma 4.2 of Titchmarsh \cite{Titchmarsh-book}.  $\hfill$
\end{proof}

\begin{lemma}\label{yijie-exp-pair}
Suppose that $f(x):[a,b]\to\mathbb{R}$ has continuous derivatives of arbitrary order on $[a,b]$, where $1\leqslant a<b\leqslant2a$. Suppose further that
\begin{equation*}
 \big|f^{(j)}(x)\big|\asymp \lambda_1 a^{1-j},\qquad j\geqslant1, \qquad x\in[a,b].
\end{equation*}
Then for any exponential pair $(\kappa,\lambda)$, we have
\begin{equation*}
 \sum_{a<n\leqslant b}e(f(n))\ll \lambda_1^\kappa a^\lambda+\lambda_1^{-1}.
\end{equation*}
\end{lemma}
\begin{proof}
 See (3.3.4) of Graham and Kolesnik \cite{Graham-Kolesnik-book}.  $\hfill$
\end{proof}

\begin{lemma}\label{Mu-lemma-Tolev-lemma}
For $1<c<2$, we have
\begin{equation}\label{S-Phi-2-ci}
  \int_{\tau<|x|<K}\big|S^2(x)\Phi(x)\big|\mathrm{d}x\ll X^{1+\eta},
\end{equation}
\begin{equation}\label{S-Phi-4-ci}
  \int_{\tau<|x|<K}\big|S^4(x)\Phi(x)\big|\mathrm{d}x\ll X^{4-c+\eta},
\end{equation}
\begin{equation}\label{S-2-major}
  \int_{-\tau}^{+\tau}\big|S(x)\big|^2\mathrm{d}x\ll X^{2-c}\log^3X,
\end{equation}
\begin{equation}\label{I-2-major}
  \int_{-\tau}^{+\tau}\big|I(x)\big|^2\mathrm{d}x\ll X^{2-c}\log^3X.
\end{equation}
\end{lemma}
\begin{proof}
 For (\ref{S-Phi-2-ci}) and (\ref{S-Phi-4-ci}), one can see Lemma 2.6 of Mu \cite{Mu-Quanwu-2015}. For (\ref{S-2-major}) and (\ref{I-2-major}), one can see
 Lemma 7 of Tolev \cite{Tolev-1992}.      $\hfill$
\end{proof}

\begin{lemma}\label{S(x)=I(x)+jieyu}
   For $1<c<2$, then for $|x|\leqslant\tau$ we have
  \begin{equation*}
     S(x)=I(x)+O\left(X\exp\big(-(\log X)^{1/5}\big)\right).
  \end{equation*}
\end{lemma}
\begin{proof}
    See Lemma 4 of Zhai and Cao \cite{Zhai-Cao-2003}.  $\hfill$
\end{proof}

\begin{lemma}\label{4-power-main-low-bound}
   For $1<c<2$, we have
   we have
     \begin{equation*}
        \int_{-\infty}^{+\infty} I^4(x)\Phi(x)e(-Nx)\mathrm{d}x\gg\varepsilon X^{4-c}.
     \end{equation*}
\end{lemma}
\begin{proof}
 See Lemma 8 of Zhai and Cao \cite{Zhai-Cao-2003}.   $\hfill$
\end{proof}

\begin{lemma}\label{Sargos-Wu-exponential-sum}
  Let $\alpha,\beta\in\mathbb{R}$ with $\alpha\beta(\alpha-1)(\beta-1)(\alpha-2)(\beta-2)\not=0, F>0, M\geqslant1, L\geqslant1, |a_m|\leqslant1, |b_\ell|\leqslant1$. Then we have
\begin{align*}
     & \,\, (FML)^{-\eta}\cdot\bigg|\sum_{M<m\leqslant2M}\sum_{L<\ell\leqslant2L}a_mb_\ell e\bigg(F\frac{m^\alpha\ell^\beta}{M^\alpha L^\beta}\bigg)\bigg|
                      \nonumber \\
\ll & \,\, \big(F^4M^{31}L^{34}\big)^{\frac{1}{42}} + \big(F^6M^{53}L^{51}\big)^{\frac{1}{66}} + \big(F^6M^{46}L^{41}\big)^{\frac{1}{56}}
            +\big(F^2M^{38}L^{29}\big)^{\frac{1}{40}}+ \big(FM^{9}L^{6}\big)^{\frac{1}{10}}
                        \nonumber \\
    & \,\,  + \big(F^2M^{7}L^{6}\big)^{\frac{1}{10}}+ \big(F^3M^{43}L^{32}\big)^{\frac{1}{46}}
            + \big(FM^{6}L^{6}\big)^{\frac{1}{8}}+ M^{\frac{1}{2}}L+ML^{\frac{1}{2}}+F^{-\frac{1}{2}}ML.
\end{align*}

\end{lemma}
\begin{proof}
 See Theorem 9 of Sargos and Wu \cite{Sargos-Wu-2000}.  $\hfill$
\end{proof}

\begin{lemma}\label{Heath-Brown-exponent-sum-fenjie}
Let $3<U<V<Z<X$ and suppose that $Z-\frac{1}{2}\in\mathbb{N},\,X\gg Z^2U,\,Z\gg U^2,\,V^3\gg X$. Assume further that $F(n)$ is a
complex--valued function such that $|F(n)|\leqslant1$. Then the sum
\begin{equation*}
 \sum_{X<n\leqslant2X}\Lambda(n)F(n)
\end{equation*}
can be written into $O(\log^{10}X)$ sums, each of which either of Type I:
\begin{equation*}
 \sum_{M<m\leqslant2M}a(m)\sum_{L<\ell\leqslant2L}F(m\ell)
\end{equation*}
with $L\gg Z$, where $a(m)\ll m^{\eta},\,ML\asymp X$, or of Type II:
\begin{equation*}
 \sum_{M<m\leqslant2M}a(m)\sum_{L<\ell\leqslant2L}b(\ell)F(m\ell)
\end{equation*}
with $U\ll M\ll V$, where $a(m)\ll m^{\eta},\,b(\ell)\ll \ell^{\eta},\,ML\asymp X$.
\end{lemma}
\begin{proof}
 See Lemma 3 of Heath--Brown \cite{Heath-Brown-1983}.  $\hfill$
\end{proof}

\begin{lemma}\label{4-variables-Type-I}
Suppose that $\tau<|x|<K,\, M\ll X^{\frac{2971}{5334}},a(m)\ll m^\eta,ML\asymp X$, then we have
\begin{equation*}
 S_I(M,L):=\sum_{M<m\leqslant2M}\sum_{L<\ell\leqslant2L}a(m)e(xm^c\ell^c)\ll X^{\frac{2515}{2667}+\eta}.
\end{equation*}
\end{lemma}
\begin{proof}
If $M\ll X^{\frac{4961}{10668}}$, then by Lemma \ref{yijie-exp-pair} with the exponential pair $(\kappa,\lambda)=A^2B(0,1)=(\frac{1}{14},\frac{11}{14})$,
 we deduce that
\begin{align*}
  S_I(M,L) \ll & \,\, X^\eta\sum_{M<m\leqslant2M}\Bigg|\sum_{L<\ell\leqslant2L}e(xm^c\ell^c)\Bigg|
                               \nonumber  \\
    \ll & \,\, X^\eta\sum_{M<m\leqslant2M}\bigg(\big(|x|X^cL^{-1}\big)^\frac{1}{14}L^{\frac{11}{14}}+\frac{1}{|x|X^cL^{-1}}\bigg)
                               \nonumber  \\
     \ll & \,\, X^\eta\Big(K^{\frac{1}{14}}X^{\frac{c}{14}}ML^{\frac{5}{7}}+\tau^{-1}X^{1-c}\Big)
                                \nonumber  \\
     \ll & \,\, X^{\frac{c}{14}+\frac{5}{7}+\eta}M^{\frac{2}{7}}\ll  X^{\frac{2515}{2667}+\eta}.
\end{align*}
If $X^{\frac{4961}{10668}}\ll M \ll X^{\frac{2971}{5334}}$, then by Lemma \ref{Sargos-Wu-exponential-sum} with $(m,\ell)=(m,\ell)$, we obtain 
\begin{equation*}
  S_I(M,L)\ll X^{\frac{2515}{2667}+\eta},
\end{equation*}
which completes the proof of Lemma \ref{4-variables-Type-I}.      $\hfill$
\end{proof}

\begin{lemma}\label{4-variables-Type-II}
Suppose that  $\tau<|x|<K,X^{\frac{304}{2667}}\ll M\ll X^{\frac{1147}{2667}}, a(m)\ll m^\eta,b(\ell)\ll \ell^\eta, ML\asymp X$. Then we have
\begin{equation*}
 \mathcal{S}_{II}(M,L):=\sum_{M<m\leqslant2M}\sum_{L<\ell\leqslant2L}a(m)b(\ell)e(xm^c\ell^c)\ll X^{\frac{2515}{2667}+\eta}.
\end{equation*}
\end{lemma}
\begin{proof}
Taking $Q=X^{\frac{304}{2667}}(\log X)^{-1}$, if $X^{\frac{304}{2667}}\ll M\ll X^{\frac{1147}{2667}}$, by
Cauchy's inequality and Lemma \ref{Fouvry-Iwaniec-chafen}, we deduce that
\begin{align}\label{Type-II-inner-sum}
   & \,\,  S_{II}(M,L) \ll   \Bigg(\sum_{L<\ell\leqslant2L}|b(\ell)|^2\Bigg)^{\frac{1}{2}}
              \Bigg(\sum_{L<\ell\leqslant2L}\Bigg|\sum_{M<m\leqslant2M}a(m)e(xm^c\ell^c)\Bigg|^2\Bigg)^{\frac{1}{2}}
                     \nonumber  \\
    \ll & \,\, L^{\frac{1}{2}+\eta} \Bigg(\sum_{L<\ell\leqslant2L}\frac{M}{Q}\sum_{0\leqslant q<Q}\bigg(1-\frac{q}{Q}\bigg)
                     \nonumber  \\
    & \qquad \qquad \times\sum_{M+q<m\leqslant2M-q}a(m+q)\overline{a(m-q)}e\Big(x\ell^c\big((m+q)^c-(m-q)^c\big)\Big)\Bigg)^{\frac{1}{2}}
                      \nonumber  \\
    \ll & \,\, L^{\frac{1}{2}+\eta}\Bigg(\frac{M}{Q}\sum_{L<\ell\leqslant2L}\bigg(M^{1+\eta}+\sum_{1\leqslant q<Q}\bigg(1-\frac{q}{Q}\bigg)
                      \nonumber  \\
    & \qquad \qquad \times\sum_{M+q<m\leqslant2M-q}a(m+q)\overline{a(m-q)}e\Big(x\ell^c\big((m+q)^c-(m-q)^c\big)\Big) \bigg)\Bigg)^{\frac{1}{2}}
                      \nonumber  \\
    \ll & \,\, X^\eta\Bigg(\frac{X^2}{Q}+\frac{X}{Q}\sum_{1\leqslant q<Q}\sum_{M<m\leqslant2M}
               \Bigg|\sum_{L<\ell\leqslant2L}e\Big(x\ell^c\big((m+q)^c-(m-q)^c\big)\Big) \Bigg|\Bigg)^{\frac{1}{2}}.
\end{align}
Therefore, it is sufficient to estimate the inner sum
\begin{equation*}
  \mathfrak{S}_0:=\sum_{L<\ell\leqslant2L}e\Big(x\ell^c\big((m+q)^c-(m-q)^c\big)\Big).
\end{equation*}
From Lemma \ref{yijie-exp-pair} with the exponential pair $(\kappa,\lambda)=AB(0,1)=(\frac{1}{6},\frac{2}{3})$, we have
\begin{equation}\label{inner-S_0}
   \mathfrak{S}_0\ll \big(|x|X^{c-1}q\big)^{\frac{1}{6}}L^{\frac{2}{3}}+\frac{1}{|x|X^{c-1}q}.
\end{equation}
Putting the estimate (\ref{inner-S_0}) into (\ref{Type-II-inner-sum}), we deduce that  
\begin{align*}
             S_{II}(M,L)
  \ll & \,\, X^\eta\Bigg(\frac{X^2}{Q}+\frac{X}{Q}\sum_{1\leqslant q<Q}
            \sum_{M<m\leqslant2M}\bigg(\big(|x|X^{c-1}q\big)^{\frac{1}{6}}L^{\frac{2}{3}}+\frac{1}{|x|X^{c-1}q}\bigg)\Bigg)^{\frac{1}{2}}
                           \nonumber  \\
  \ll & \,\, X^\eta\Bigg(\frac{X^2}{Q}+\frac{X}{Q}\Big(K^\frac{1}{6}X^{\frac{c-1}{6}}L^{\frac{2}{3}}MQ^{\frac{7}{6}}
             +\tau^{-1}X^{1-c}M\log Q\Big)\Bigg)^{\frac{1}{2}}
                             \nonumber  \\
   \ll & \,\, \big(X^{2+\eta}Q^{-1}\big)^{\frac{1}{2}}\ll X^{\frac{2515}{2667}+\eta},
\end{align*}
which completes the proof of Lemma \ref{4-variables-Type-II}.   $\hfill$
\end{proof}

\begin{lemma}\label{S(x)-yuqujianguji-4}
For $1<c<\frac{1193}{889}$ and $\tau<|x|<K$, we have
\begin{equation*}
 S(x)\ll X^{\frac{2515}{2667}+\eta}.
\end{equation*}
\end{lemma}
\begin{proof}
Trivially, we have
\begin{equation}\label{S=U+error-4}
  S(x)=\mathfrak{U}(x)+O(X^{1/2}),
\end{equation}
where
\begin{equation*}
  \mathfrak{U}(x)=\sum_{X<n\leqslant2X}\Lambda(n)e(n^cx).
\end{equation*}
Taking $U=X^{\frac{304}{2667}},V=X^{\frac{1147}{2667}},Z=\big[X^{\frac{2363}{5334}}\big]+\frac{1}{2}$ in Lemma \ref{Heath-Brown-exponent-sum-fenjie},
it is not difficult to see that the sum
\begin{equation*}
 \sum_{X<n\leqslant2X}\Lambda(n)e(n^cx)
\end{equation*}
can be written into $O(\log^{10}X)$ sums, each of which either of Type I:
\begin{equation*}
  S_{I}(M,L)=\sum_{M<m\leqslant2M}\sum_{L<\ell\leqslant2L}a(m)e(xm^c\ell^c)
\end{equation*}
with $L\gg Z, a(m)\ll m^\eta, ML\asymp X$, or of Type II:
\begin{equation*}
  S_{II}(M,L)=\sum_{M<m\leqslant2M}\sum_{L<\ell\leqslant2L}a(m)b(\ell)e(xm^c\ell^c)
\end{equation*}
with $U\ll M\ll V,a(m)\ll m^\eta,b(\ell)\ll \ell^\eta, ML\asymp X$. For the sums of Type I, since $L\gg Z$ and $ML\asymp X$, 
we get $M\ll X^{\frac{2971}{5334}}$. By Lemma \ref{4-variables-Type-I}, we have $S_I(M,L)\ll X^{\frac{2515}{2667}+\eta}$. For the sums of Type II, 
by Lemma \ref{4-variables-Type-II},  we get $S_{II}(M,L)\ll X^{\frac{2515}{2667}+\eta}$. Thus, we deduce that
\begin{equation}\label{U-guji-4}
 \sum_{X<n\leqslant2X}\Lambda(n)e(n^cx)\ll X^{\frac{2515}{2667}+\eta}.
\end{equation}
From (\ref{S=U+error-4}) and (\ref{U-guji-4}), we finish the proof of Lemma \ref{S(x)-yuqujianguji-4}.  $\hfill$
\end{proof}

\section{Proof of Theorem \ref{Theorem-four-primes-variables} }

In this section, we use $\Phi(x)$ and $\phi(y)$ to denote the functions which appear in Lemma \ref{xiaobei-lemma} with parameter $a=\frac{9\varepsilon}{10},b=\frac{\varepsilon}{10},r=[\log X]$. Define
\begin{equation*}
  \mathscr{B}_4(N)=\sum_{\substack{X<p_1,p_2,p_3,p_4\leqslant2X \\ |p_1^c+\cdots+p_4^c-N|<\varepsilon }}\prod_{j=1}^4\log p_j.
\end{equation*}
From the property of $\phi(y)$, we get $\mathscr{B}_4(N)\geqslant \mathscr{C}_{4}(N)$, where
\begin{equation*}
  \mathscr{C}_4(N)=\sum_{X<p_1,p_2,p_3,p_4\leqslant2X }\Bigg(\prod_{j=1}^4\log p_j\Bigg)\phi(p_1^c+\cdots+p_4^c-N).
\end{equation*}
From the Fourier transformation formula, we derive that
\begin{align}\label{C_5(N)-fenie}
 \mathscr{C}_4(N)= & \,\, \sum_{X<p_1,\dots,p_4\leqslant2X}\Bigg(\prod_{j=1}^4\log p_j\Bigg)
                         \int_{-\infty}^{+\infty}e\big((p_1^c+\cdots+p_4^c-N)y\big)\Phi(y)\mathrm{d}y
                               \nonumber \\
 = & \,\,  \int_{-\infty}^{+\infty}S^4(x)\Phi(x)e(-Nx)\mathrm{d}x
                               \nonumber \\
 =& \,\, \bigg(\int_{|x|\leqslant\tau}+\int_{\tau<|x|<K}+\int_{|x|\geqslant K}\bigg)S^4(x)\Phi(x)e(-Nx)\mathrm{d}x
                                \nonumber \\
 =& \,\, \mathscr{C}_4^{(1)}(N)+\mathscr{C}_4^{(2)}(N)+\mathscr{C}_4^{(3)}(N),\quad \textrm{say}.
\end{align}
\subsection{The Estimate of $\mathscr{C}_4^{(1)}(N)$}
Define
\begin{align*}
  & \,\, \mathscr{H}_4(N)=\int_{-\infty}^{+\infty} I^4(x)\Phi(x)e(-Nx)\mathrm{d}x,  \\
  & \,\, \mathscr{H}_\tau(N)=\int_{-\tau}^{+\tau} I^4(x)\Phi(x)e(-Nx)\mathrm{d}x.
\end{align*}
From Lemma \ref{xiaobei-lemma} and Lemma \ref{Titchmarsh-42}, we derive that
\begin{align}\label{H_tau-H-4}
  \big|\mathscr{H}_4(N)-\mathscr{H}_\tau(N)\big|\ll \int_{\tau}^\infty|I(x)|^4|\Phi(x)|\mathrm{d}x
              \ll \varepsilon\int_{\tau}^{\infty}\bigg(\frac{1}{|x|X^{c-1}}\bigg)^4\mathrm{d}x\ll  \varepsilon X^{4-c-\eta}.
\end{align}
From Lemma \ref{Mu-lemma-Tolev-lemma}, Lemma \ref{S(x)=I(x)+jieyu} and the trivial estimate $S(x)\ll X$, we get
\begin{align}\label{C_4^(1)-H_tau}
   & \,\, \big|\mathscr{C}_4^{(1)}(N)-\mathscr{H}_\tau(N)\big|\leqslant\int_{-\tau}^{+\tau}\big|S^4(x)-I^4(x)\big|\big|\Phi(x)\big|\mathrm{d}x
                     \nonumber \\
   \ll & \,\, \varepsilon\cdot\int_{-\tau}^{+\tau}\big|S(x)-I(x)\big|\big(|S(x)|^3+|I(x)|^3\big)\mathrm{d}x
                     \nonumber \\
   \ll & \,\, \varepsilon\cdot X\exp\big(-(\log X)^{1/5}\big)\bigg(\int_{-\tau}^{+\tau}|S(x)|^3\mathrm{d}x+\int_{-\tau}^{+\tau}|I(x)|^3\mathrm{d}x\bigg)
                     \nonumber \\
   \ll & \,\, \varepsilon X^{4-c}\exp\big(-(\log X)^{1/6}\big).
\end{align}
It follows from Lemma \ref{4-power-main-low-bound}, (\ref{H_tau-H-4}) and (\ref{C_4^(1)-H_tau}) that
\begin{equation}\label{C_4^(1)(N)-lower-bound}
  \mathscr{C}_4^{(1)}(N)=\big( \mathscr{C}_4^{(1)}(N)-\mathscr{H}_\tau(N)\big)+\big(\mathscr{H}_\tau(N)-\mathscr{H}_4(N)\big)+\mathscr{H}_4(N)
                       \gg \varepsilon  X^{4-c}.
\end{equation}

\subsection{The Estimate of $\mathscr{C}_4^{(2)}(N)$}
According to the definition of $\mathscr{C}_4^{(2)}(N)$, we obtain
\begin{align*}
  \big|\mathscr{C}_4^{(2)}(N)\big|=& \,\, \Bigg|\sum_{X<p\leqslant2X}(\log p)\int_{\tau<|x|<K}e(p^cx)S^3(x)\Phi(x)e(-Nx)\mathrm{d}x\Bigg|
                     \nonumber \\
  \leqslant & \,\, \sum_{X<p\leqslant2X}(\log p)\Bigg|\int_{\tau<|x|<K}e(p^cx)S^3(x)\Phi(x)e(-Nx)\mathrm{d}x\Bigg|
                     \nonumber \\
  \ll & \,\,(\log X)\sum_{X<n\leqslant2X}\Bigg|\int_{\tau<|x|<K}e(n^cx)S^3(x)\Phi(x)e(-Nx)\mathrm{d}x\Bigg|.
\end{align*}
By Cauchy's inequality, we deduce that
\begin{align}\label{C_4^(2)-Cauchy-upper}
          &\,\,     \big|\mathscr{C}_4^{(2)}(N)\big|
   \ll  \, X^{\frac{1}{2}}(\log X)\Bigg(\sum_{X<n\leqslant2X}\Bigg|\int_{\tau<|x|<K}e(n^cx)S^3(x)\Phi(x)e(-Nx)\mathrm{d}x\Bigg|^2\Bigg)^{\frac{1}{2}}
                       \nonumber \\
   = & \, X^{\frac{1}{2}}(\log X)\Bigg(\sum_{X<n\leqslant2X}\int_{\tau<|x|<K}e(n^cx)S^3(x)\Phi(x)e(-Nx)\mathrm{d}x
                            \nonumber \\
   & \,\, \qquad \qquad\qquad\qquad\qquad\qquad\times\int_{\tau<|y|<K}\overline{e(n^cy)S^3(y)\Phi(y)e(-Ny)}\mathrm{d}y\Bigg)^{\frac{1}{2}}
                            \nonumber \\
   = & \, X^{\frac{1}{2}}(\log X)\Bigg(\int_{\tau<|y|<K}\!\!\!\overline{S^3(y)\Phi(y)e(-Ny)}\mathrm{d}y
                 \int_{\tau<|x|<K}\!\!\!S^3(x)\Phi(x)e(-Nx)\mathcal{T}(x-y)\mathrm{d}x\Bigg)^{\frac{1}{2}}
                            \nonumber \\
   \ll & \,\, X^{\frac{1}{2}}(\log X)\Bigg(\int_{\tau<|y|<K}\big|S^3(y)\Phi(y)\big|\mathrm{d}y
             \int_{\tau<|x|<K}\big|S^3(x)\Phi(x)\mathcal{T}(x-y)\big|\mathrm{d}x\Bigg)^{\frac{1}{2}}.
\end{align}
For the inner integral in (\ref{C_4^(2)-Cauchy-upper}), we get
\begin{align}\label{4-inner-integral-fenjie}
  & \,\, \int_{\tau<|x|<K}\big|S^3(x)\Phi(x)\mathcal{T}(x-y)\big|\mathrm{d}x
                        \nonumber \\
  \ll & \,\, \int_{\substack{\tau<|x|<K \\ |x-y|\leqslant X^{-c}}}\big|S^3(x)\Phi(x)\mathcal{T}(x-y)\big|\mathrm{d}x
               +\int_{\substack{\tau<|x|<K \\ X^{-c}<|x-y|\leqslant2K}}\big|S^3(x)\Phi(x)\mathcal{T}(x-y)\big|\mathrm{d}x.
\end{align}
From Lemma \ref{xiaobei-lemma}, Lemma \ref{S(x)-yuqujianguji-4} and the trivial estimate $\mathcal{T}(x-y)\ll X$, we get 
\begin{align}\label{4-inner-integral-fenjie-1}
  & \,\, \int_{\substack{\tau<|x|<K \\ |x-y|\leqslant X^{-c}}}\big|S^3(x)\Phi(x)\mathcal{T}(x-y)\big|\mathrm{d}x
                               \nonumber \\
   \ll & \,\, \varepsilon X\times \sup_{\tau<|x|<K} |S(x)|^3\times   \int_{\substack{\tau<|x|<K \\ |x-y|\leqslant X^{-c}}} \mathrm{d}x
                                \nonumber \\
   \ll & \,\, \varepsilon X\cdot X^{\frac{2515}{889}-c+\eta}\ll \varepsilon X^{\frac{3404}{889}-c+\eta}.
\end{align}
According to Lemma \ref{yijie-exp-pair}, for $X^{-c}<|x-y|\leqslant2K$, we get
\begin{align}\label{T(x)-expo-pair-abs-4}
 \mathcal{T}(x-y) \ll & \,\, \big(|x-y|X^{c-1}\big)^\kappa X^\lambda+\frac{1}{|x-y|X^{c-1}}
                                   \nonumber \\
   \ll & \,\, |x-y|^\kappa X^{\kappa c+\lambda-\kappa}+\frac{1}{|x-y|X^{c-1}} .
\end{align}
By choosing 
\begin{equation*}
  (\kappa,\lambda)= BA^3 BA^2  BA   BA BA^2 BA  BA B(0,1)=\bigg(\frac{1731}{4492},\frac{591}{1123}\bigg)
\end{equation*}
in (\ref{T(x)-expo-pair-abs-4}), we deduce that
\begin{align}\label{T(x)-expo-pair-explicit-4}
 \mathcal{T}(x)
   \ll & \,\, |x-y|^\frac{1731}{4492} X^{\frac{1731c+633}{4492}}+\frac{1}{|x-y|X^{c-1}} .
\end{align}
On the other hand, by Lemma \ref{Mu-lemma-Tolev-lemma} and Cauchy's inequality, we obtain
\begin{align}\label{S^3(x)-mean}
     & \,\,  \int_{\tau<|x|<K}\big|S^3(x)\Phi(x)\big|\mathrm{d}x
                    \nonumber \\
\ll & \,\,\bigg(\int_{\tau<|x|<K}\big|S^2(x)\Phi(x)\big|\mathrm{d}x\bigg)^{\frac{1}{2}}
           \bigg(\int_{\tau<|x|<K}\big|S^4(x)\Phi(x)\big|\mathrm{d}x\bigg)^{\frac{1}{2}}
                    \nonumber \\
\ll & \big(X^{1+\eta}\big)^{\frac{1}{2}}\cdot\big(X^{4-c+\eta}\big)^{\frac{1}{2}}\ll X^{\frac{5-c}{2}+\eta}.
\end{align}
By (\ref{T(x)-expo-pair-explicit-4}), (\ref{S^3(x)-mean}) and Lemma \ref{S(x)-yuqujianguji-4}, we have
\begin{align}\label{4-inner-integral-fenjie-2}
 & \,\, \int_{\substack{\tau<|x|<K \\ X^{-c}<|x-y|\leqslant2K}}\big|S^3(x)\Phi(x)\mathcal{T}(x-y)\big|\mathrm{d}x
                     \nonumber \\
 \ll & \,\, \int_{\substack{\tau<|x|<K \\ X^{-c}<|x-y|\leqslant2K}}\big|S^3(x)\Phi(x)\big|
              \bigg(|x-y|^\frac{1731}{4492} X^{\frac{1731c+633}{4492} }+\frac{1}{|x-y|X^{c-1}} \bigg)\mathrm{d}x
                     \nonumber \\
 \ll & \,\, X^{\frac{1731c+633}{4492}+\eta}\int_{\tau<|x|<K}\!\!\big|S^3(x)\Phi(x)\big|\mathrm{d}x
                      \nonumber \\
    & \,\,  \qquad+\varepsilon X^{1-c}\times \sup_{\tau<|x|<K}|S(x)|^3\times \int_{\substack{\tau<|x|<K \\ X^{-c}<|x-y|\leqslant2K}}\frac{\mathrm{d}x}{|x-y|}
                        \nonumber \\
 \ll & \,\, X^{\frac{1731c+633}{4492}+\eta}\cdot X^{\frac{5-c}{2}+\eta}+\varepsilon X^{1-c}\cdot X^{\frac{2515}{889}+\eta}
                       \nonumber \\
  \ll & \,\,X^{\frac{11863-515c}{4492}+\eta}+ \varepsilon X^{\frac{3404}{889}-c+\eta} \ll \varepsilon X^{\frac{3404}{889}-c+\eta}.
\end{align}
From (\ref{4-inner-integral-fenjie}), (\ref{4-inner-integral-fenjie-1}) and (\ref{4-inner-integral-fenjie-2}), we get
\begin{equation*}
  \int_{\tau<|x|<K}\big|S^3(x)\Phi(x)\mathcal{T}(x-y)\big|\mathrm{d}x\ll\varepsilon X^{\frac{3404}{889}-c+\eta},
\end{equation*}
from which and (\ref{S^3(x)-mean}), we can conclude that
\begin{equation}\label{C_4^(2)(N)-upper-bound}
  \big|\mathscr{C}_4^{(2)}(N)\big|\ll X^{\frac{1}{2}}(\log X)\Big(X^{\frac{5-c}{2}+\eta}\cdot \varepsilon X^{\frac{3404}{889}-c+\eta}\Big)^{\frac{1}{2}}
    \ll \varepsilon X^{4-c-\eta}.
\end{equation}

\subsection{The Estimate of $\mathscr{C}_4^{(3)}(N)$}
According to Lemma \ref{xiaobei-lemma}, we have
\begin{align}\label{C_4^(3)(N)-upper-bound}
        \big|\mathscr{C}_4^{(3)}(N)\big|\ll & \,\, \int_{K}^\infty |S(x)|^{4}|\Phi(x)|\mathrm{d}x
           \ll   X^4\int_{K}^\infty\frac{1}{\pi|x|}\bigg(\frac{r}{2\pi|x|b}\bigg)^r\mathrm{d}x
                     \nonumber \\
        \ll & \,\, X^4\bigg(\frac{r}{2\pi b}\bigg)^r\int_K^\infty\frac{\mathrm{d}x}{x^{r+1}}\ll \frac{X^4}{r}\bigg(\frac{r}{2\pi Kb}\bigg)^r
                      \nonumber \\
        \ll  & \,\,\frac{X^4}{\log X}\cdot\bigg(\frac{1}{2\pi\log^7X}\bigg)^{\log X}
                   \ll \frac{X^4}{X^{7\log\log X+\log(2\pi)}}\ll 1.
\end{align}

\subsection{Proof of Theorem \ref{Theorem-four-primes-variables}}
From (\ref{C_5(N)-fenie}), (\ref{C_4^(1)(N)-lower-bound}), (\ref{C_4^(2)(N)-upper-bound}) and (\ref{C_4^(3)(N)-upper-bound}), we deduce that
\begin{equation*}
\mathscr{C}_4(N)=\mathscr{C}_4^{(1)}(N)+\mathscr{C}_4^{(2)}(N)+\mathscr{C}_4^{(3)}(N)\gg \varepsilon X^{4-c},
\end{equation*}
and thus
\begin{equation*}
\mathscr{B}_4(N)  \geqslant\mathscr{C}_4(N) \gg \varepsilon X^{4-c}\gg \frac{X^{4-c}}{\log^2X},
\end{equation*}
which completes the proof of Theorem \ref{Theorem-four-primes-variables}.

\section*{Acknowledgement}


The authors would like to express the most sincere gratitude to Professor Wenguang Zhai for his valuable advices and constant encouragement.

\end{document}